\documentclass[11pt]{amsart}

\usepackage{amsmath,amsthm,amssymb,amsfonts}
\usepackage{tikz}
\usetikzlibrary{calc}
\usepackage{setspace}
\usepackage[margin=1.2in]{geometry}

\begin{document} 

\newcommand*{\longerrightarrow}{\ensuremath{\relbar\joinrel\relbar\joinrel\relbar\joinrel\rightarrow}}

\pagestyle{plain}

\theoremstyle{plain}
\newtheorem{thm}{Theorem}[section]
\newtheorem{lem}[thm]{Lemma}
\newtheorem{prop}[thm]{Proposition}
\newtheorem{cor}[thm]{Corollary}

\theoremstyle{definition}
\newtheorem{defn}[thm]{Definition}
\newtheorem{conj}[thm]{Conjecture}
\newtheorem{eg}[thm]{Example}

\theoremstyle{remark}
\newtheorem{rem}[thm]{Remark}

\title{Morse Matchings on a Hypersimplex}
\author{Jacob T. Harper}
\address{Department of Mathematics\\
					University of Colorado\\
					Campus Box 395\\
					Boulder, CO 80309-0395\\
					USA				}
\email{Jacob.Harper@colorado.edu}					

\begin{abstract}
We present a family of complete acyclic Morse matchings on the face lattice of a hypersimplex. Since a hypersimplex is a convex polytope, there is a natural way to form a CW complex from its faces. In a future paper we will utilize these matchings to classify every subcomplex whose reduced homology groups are concentrated in a single degree and describe a homology basis for each of them.
\end{abstract}

\maketitle

\doublespace

\section{Introduction}
The standard $n$-simplex is the polytope whose vertices are the $(0,1)$-vectors in $\mathbb{R}^{n+1}$ that sum to 1. A hypersimplex is a generalization of the $n$-simplex whose vertices are the same vectors that instead sum to $k$ where $1\leq k\leq n$. 
The subject of study in this paper is focused around this family of polytopes and more specifically, their face lattices.

In previous work \cite{Green09}, Green classified the faces of a different polytope, known as the half cube. They assemble naturally into a regular CW complex, $C_n$. This complex contained an interesting subcomplex $C_{n,k}$, obtained by deleting the interiors of all the half cube shaped faces of dimension $l\geq k$. It was also shown in \cite[Theorem 3.3.2]{Green09} that the reduced homology of $C_{n,k}$ is free over $\mathbb{Z}$ and concentrated in degree $k-1$ by utilizing a partial acyclic matching on the Hasse diagram of the face lattice of $C_{n,k}$.

The Coxeter group $W(D_n)$ acts naturally on the $(k-1)$-st homology of $C_{n,k}$. Using results from \cite{Green09}, Green then computed the character of this representation (over $\mathbb{C}$) in \cite[Theorem 4.4]{Green10}. Finally, a construction of a complete acyclic matching was given in \cite{GH1} and this led to a description of an explicit basis for the $(k-1)$-st homology of $C_{n,k}$.

In Section \ref{s:Construction}, we continue the work started in \cite{GH1}
by constructing a family of complete matchings on the Hasse diagram corresponding to the faces of a hypersimplex. 
The main result of this paper (Theorem \ref{thm:acyc}) then shows that these matchings are acyclic.
It should be noted that the existence of such a matching is neither surprising nor the point.
In a future paper we will show how the construction, and not the existence, allows us to give
a description of the characters of the representation that arises from the Coxeter group $W(A_n)$ acting naturally on the hypersimplex, classify subcomplexes whose reduced homology groups are concentrated in a single degree, and describe an explicit basis for these homology groups.
We also believe that finding such a matching on the face lattices of polytopes is an aesthetically pleasing goal in its own right.
The work here is based on the Ph.D. thesis of the author \cite{HarperPHD}, which was directed by R.M. Green.

\section{The Hypersimplex}

In this section we introduce some notation and explore some of the properties of a hypersimplex that will be needed in the subsequent sections.

\begin{defn}
A subset $K\subset\mathbb{R}^n$ is said to be \textbf{convex} if the straight line segment between any two points in $K$ is also contained in $K$.
For any $K\subset\mathbb{R}^n$, the smallest convex set containing $K$ can be constructed as the intersection of all convex sets that contain $K$ and is known as the \textbf{convex hull} of $K$. A \textbf{polytope} is the convex hull of a finite set of points in $\mathbb{R}^n$.
\end{defn}

Note that if $H$ is a hyperplane in $\mathbb{R}^n$, then the complement $\mathbb{R}^n-H$ has two open components. A closed halfspace is the union of one of those two components with the hyperplane. A polytope can also be defined as the bounded intersection of finitely many closed halfspaces. It is nontrivial that these two definitions for a polytope are equivalent, but it is shown in the following theorem to be true.

\begin{thm}[\textbf{The Main Theorem for Polytopes}]
A subset $P\subset\mathbb{R}^n$ is the convex hull of a finite subset if and only if it is a bounded intersection of halfspaces.
\end{thm}
\begin{proof}
This is \cite[Theorem 1.1]{Ziegler95}. 
\end{proof}

Both definitions are helpful as some properties of polytopes are easier to understand when viewing them one way versus the other. For example, the following definition relates more to the halfspace definition of a polytope and can be found in Ziegler's book \cite{Ziegler95}.

\begin{defn}
Let $P\subseteq \mathbb{R}^n$ be a polytope. A \textbf{face} $F$ of $P$ is either $P$ itself or the intersection of $P$ with a hyperplane such that $P$ is contained in one of the two closed halfspaces determined by the hyperplane. The faces of dimension 0, 1, and $i$ are called the \textbf{vertices}, \textbf{edges}, and \textbf{$i$-faces} of $P$ respectively.
\end{defn}

In this paper we will investigate the properties of the following polytope.

\begin{defn}
Let $J(n,k)$ be the polytope equal to the convex hull of the points in $\mathbb{R}^n$ with exactly $k$ 1's and $n-k$ 0's. For any value of $n$ and $k$ such that $1\leq k\leq n-1$, this polytope is called a \textbf{hypersimplex}.
\end{defn}

It turns out that every face of a polytope is equal to the convex hull of a subset of the vertices. We will use the following definition to fully describe the face lattice of the hypersimplex.

\begin{defn}
Let $S$ be a sequence of length $n$ made up of 0's, 1's, and~*'s. Define $F(S)$ to be the face of $J(n,k)$ equal to the convex hull of the set of vertices of this polytope whose coordinates written out in a sequence differ from $S$ only where $S$ has a *. Define $S(0)$ and $S(1)$ to be the number of 0's and 1's in the sequence $S$ respectively.
\end{defn}

\begin{eg}
Let $S$ = 11***00 with $n = 7$ and $k = 3$. Then $F(S)$ is equal to the convex hull of $(1,1,1,0,0,0,0)$, $(1,1,0,1,0,0,0)$, and $(1,1,0,0,1,0,0)$.
\end{eg}

\begin{prop}
\label{prop:faces}
The faces of $J(n,k)$ are as follows:
\begin{enumerate}
\item [\rm(i)] $\binom{n}{k}$ 0-faces (vertices) given by $\{(x_1, x_2,\ldots,x_n)\in\mathbb{R}^n \ | \ x_j\in\{0,1\}$ and $\displaystyle\sum_{j=1}^{n} x_j=k\}$;
\item [\rm(ii)] $\binom{n}{i+1}$$\binom{n-i-1}{j-1}$ i-faces $F(S)$ for each $j$ such that $1\leq j\leq k\leq j+i-1\leq n-1$, $S(0)=n-j-i$, and $S(1)=j-1$.
\end{enumerate}
\end{prop}

\begin{proof}
One proof of this can be found in \cite[Proposition 1.3.10]{HarperPHD} and is likely to be well-known.
\end{proof}


\section{CW complexes and Cellular Homology}
\label{s:CW}

In this section we show how to naturally construct a CW complex from a hypersimplex in a way that generalizes to any convex polytope. 
We start by recalling several definitions.

\begin{defn}
\
\begin{enumerate}
\item [(i)] An \textbf{$n$-cell}, $e = e^n$ is a homeomorphic copy of the open $n$-disk $D^n - S^{n-1}$, where $D^n$ is the closed unit ball in Euclidean $n$-space and $S^{n-1}$ is its boundary, the unit $(n-1)$-sphere. We call $e$ a \textbf{cell} if it is an $n$-cell for some $n$ and define dim$(e)=n$.
\item [(ii)] If a topological space $X$ is a disjoint union of cells $X=\bigcup\{e \ | \ e\in E\}$, then for each $k\geq 0 $, we define the \textbf{$k$-skeleton} $X^{(k)}$ of $X$ by $X^{(k)}=\bigcup\{e\in E \ | \ $dim$(e)\leq k \}$.
\end{enumerate}
\end{defn}


\begin{defn}
A \textbf{finite CW complex} is an ordered triple $(X,E,\Phi)$, where $X$ is a Hausdorff space, $E$ is a family of cells in $X$, and $\{\Phi_e \ | \ e\in E\}$ is a family of maps, such that
\begin{enumerate}
\item [(i)] $X=\bigcup\{e \ | \ e\in E\}$ is a disjoint union;
\item [(ii)] for each $k$-cell $e\in E$, the map $\Phi_e : D^k \longrightarrow e\cup X^{(k-1)}$ is a continuous map such that $\Phi_e(S^{k-1})\subseteq X^{(k-1)}$ and $\Phi_e|_{D^k-S^{k-1}}:D^k-S^{k-1} \longrightarrow e$ is a homeomorphism.
\end{enumerate}
If the maps $\Phi_e$ are all homeomorphisms, the CW complex is called \textbf{regular}. In this paper, we will only consider CW complexes that are regular.
\end{defn}

\begin{defn}
\label{def:subcomplex}
A \textbf{subcomplex} of the CW complex $(X,E,\Phi)$ is a triple $(|E'|,E',\Phi')$,  where $E'\subset E$, 
$|E'|:=\bigcup\{e\ | \ e\in E'\}\subset X,$
$\Phi'=\{\Phi_e \ | \ e\in E'\}$, and Im $\Phi_e \subset |E'|$ for every $e\in E'$. 
\end{defn}

\begin{prop}
Let $K$ be the hypersimplex $J(n,k)$ regarded as a subspace of $\mathbb{R}^n$, and let $E$ be the union of the following two sets:
\begin{enumerate}
\item [\rm(i)] the set of vertices of $J(n,k)$;
\item [\rm(ii)] the set of interiors of all $i$-faces of $J(n,k)$ for all $0<i<n$.
\end{enumerate}
Then $(K,E,\Phi)$ is a regular CW complex, where the maps $\Phi_e$ are the natural identifications.
\end{prop}

\begin{proof}
It is a standard result that the faces of a convex polytope form a regular CW complex; more details can be found in \cite[\S 1.3]{Forman04}.
\end{proof}

Later on it will be necessary to compute the homology groups that arise from subcomplexes of the CW complex corresponding to $J(n,k)$. 
Cellular homology is a convenient theory for doing exactly that.
We do not recall the full definition here, but instead direct the reader to \cite[\S 2.2]{Hatcher} and \cite[Proposition 5.3.10]{Geoghegan} for more details.

\begin{defn}
A \textbf{chain complex} is a sequence of abelian groups or modules $(C_i)_{i\in \mathbb{Z}}$, connected by homomorphisms (called boundary operators) $\partial_i : C_i\rightarrow C_{i-1}$, such that the composition of any two consecutive maps is zero: $\partial_i \circ \partial_{i+1} = 0$ for all $i$.
\end{defn}

Let $X$ be a regular CW complex and consider the chain complex $(C_i)_{i\in \mathbb{Z}}$, called the cellular chain complex of $X$, having the following properties.
The groups $C_i$ are all free abelian and have basis in one-to-one correspondence with the $i$-cells of $X$. The maps $\partial_i$ have the form 
$$\partial_i(e_\tau) = \displaystyle\sum_{\sigma\in X^{(i-1)}}[\tau:\sigma]e_\sigma$$ 
where $[\tau:\sigma] = 0$ if $\sigma$ is not a face of $\tau$ and $[\tau:\sigma] = \pm 1$ if $\sigma$ is a face of $\tau$, dependent on an orientation of the faces. The homology groups $H_i(X) =$ ker$(\partial_i) / $Im$(\partial_{i+1})$ are called the cellular homology groups and are equivalent to the homology groups obtained using singular homology. 
It should also be noted that we can extend $X$ by considering $\emptyset$ as a unique cell of dimension $-1$. The resulting cellular chain complex then leads to the reduced homology groups which are denoted by $\widetilde{H}_i(X)$.


\section{Discrete Morse Theory}
\label{s:Morse}

Even though cellular homology simplifies finding the homology groups by some amount, it would make things much easier if our CW complex always had the property that no two of its cells were in adjacent dimensions. 
In general this does not happen and we can not just modify a CW complex and expect it to produce the same homology groups. However the techniques introduced in this section, which were invented by Forman \cite{Forman02}, will give us a CW complex that is homotopic to ours and that have this desired property in most of the cases later on.

\begin{defn}
Let $K$ be a finite regular CW complex. A \textbf{discrete vector field} on $K$ is a collection of pairs of cells $(K_1, K_2)$ such that
\begin{enumerate}
\item [(i)] $K_1$ is a face of $K_2$ of codimension 1 and
\item [(ii)] every cell of $K$ lies in at most one such pair.
\end{enumerate}
We call a cell of $K$ \textbf{matched} if it lies in one of the above pairs and \textbf{unmatched} otherwise.
\end{defn}

\begin{defn}
If $V$ is a discrete vector field on a regular CW complex $K$, a \textbf{$V$-path} is a sequence of cells 
\begin{center}
$a_0, b_0, a_1, b_1, a_2,\ldots,b_r,a_{r+1}$
\end{center}
such that for each $i=0,\ldots, r$, $a_i$ and $a_{i+1}$ are each a codimension 1 face of $b_i$, each of the pairs $(a_i, b_i)$ belongs to $V$ (hence $a_i$ is matched with $b_i$), and $a_i \neq a_{i+1}$ for all $0\leq i \leq r$. If $r\geq 0$, we call the $V$-path \textbf{nontrivial} and if $a_0 = a_{r+1}$, we call the $V$-path \textbf{closed}.
\end{defn}

Note that all of the faces $a_i$ in the sequence above have the same dimension, $p$ say, and all of the faces $b_i$ have dimension $p+1$.

\begin{rem}
\label{rem:hasse}
Let $V$ be a discrete vector field on a regular CW complex $K$. Consider the set of cells of $K$ together with the empty cell $\emptyset$, which we consider as a cell of dimension $-1$. This gives us a partially ordered set ordered under inclusion. We can create a directed Hasse diagram, $H(V)$, from this by pointing all of the edges towards the larger cell and then reversing the direction of any edge in which the smaller cell is matched with the larger cell.
\end{rem}

\begin{eg}
\label{eg:Hasse0}
Consider the hypersimplex $J(3,1)$. This is just a $2$-dimensional simplex equal to the convex hull of the points $(1,0,0), (0,1,0),$ and $(0,0,1)$ in $\mathbb{R}^3$. Label the vertices by $A, B,$ and $C$, label the edges as $AB, AC,$ and $AB$ where the edge $AB$ is the edge joining $A$ and $B$ and likewise for the other two, and label the unique 2-face of the simplex by $ABC$ as shown on the left of Figure \ref{fig:Hasse0} below.

Let $V_1$ be the discrete vector field made up of the following pairs of faces:
\begin{center}
$(A, AB), (B, BC),$ and $(C, AC).$
\end{center}
The corresponding directed Hasse diagram $H(V_1)$ can also be seen in Figure \ref{fig:Hasse0}. Note that this is only a partial matching since the faces $ABC$ and $\emptyset$ are unmatched. The $V_1$-path $A, AB, B$ is a nontrivial $V_1$-path, but is not closed while the $V_1$-path $A, AB, B, BC, C, AC, A$ is a nontrivial closed $V_1$-path.
\end{eg}

\begin{figure}[htbp]
\begin{center}
\begin{tikzpicture}[scale=1]

\def \tscale{3}

\small

\node (A) at (0,0) [circle, draw, inner sep=1pt] {};
\node (B) at (\tscale* .5, \tscale * 3^.5/2) [circle, draw, inner sep=1pt] {};
\node (C) at (\tscale,0) [circle, draw, inner sep=1pt] {};
\node at (A) [left] {A};
\node at (B) [above] {B};
\node at (C) [right] {C};

\node (D) at (\tscale* .5, \tscale * 3^.5/4 * .9) [below] {ABC};

\fill [black] (A) circle (1.2pt);
\fill [black] (B) circle (1.2pt);
\fill [black] (C) circle (1.2pt);

\draw (A) -- node [left] {AB} (B);
\draw (B) -- node [right] {BC} (C);
\draw (A) -- node [below] {AC} (C);

\normalsize

\node at (\tscale/2,-1) {$J(3,1)$};

\end{tikzpicture}
\hspace{1 in}
\begin{tikzpicture}[scale=1, >=stealth, thick, shorten >=2pt, shorten <=2pt]
\tikzset{shade t/.style ={gray!90}}
\tikzset{nodecirc/.style ={circle, draw, inner sep=1pt, fill=black}}
\node (e) at (0,0) [nodecirc] {};
\node (A) at (-1,1) [nodecirc] {};
\node (B) at (0,1) [nodecirc] {};
\node (C) at (1,1) [nodecirc] {};
\node (AB) at (-1,2) [nodecirc] {};
\node (AC) at (0,2) [nodecirc] {};
\node (BC) at (1,2) [nodecirc] {};
\node (ABC) at (0,3) [nodecirc] {};

\footnotesize
\node at (e) [below] {$\emptyset$};
\node at (A) [left] {A};
\node at (0,.95) [right] {B};
\node at (C) [right] {C};
\node at (AB) [left] {AB};
\node at (0,2.05) [right] {AC};
\node at (BC) [right] {BC};
\node at(ABC) [above] {ABC};
\normalsize

\draw [->] [shade t] (e) -- (A);
\draw [->] [shade t] (e) -- (B);
\draw [->] [shade t] (e) -- (C);
\draw [->] [shade t] (A) -- (AC);
\draw [<-] [shade t] (A) -- (AB);
\draw [->] [shade t] (B) -- (AB);
\draw [<-] [shade t] (B) -- (BC);
\draw [<-] [shade t] (C) -- (AC);
\draw [->] [shade t] (C) -- (BC);
\draw [->] [shade t] (AC) -- (ABC);
\draw [->] [shade t] (AB) -- (ABC);
\draw [->] [shade t] (BC) -- (ABC);

\node at (0,-1) {$H(V_1)$};

\end{tikzpicture}
\caption{Example of a Hasse diagram}
\label{fig:Hasse0}
\end{center}
\end{figure}
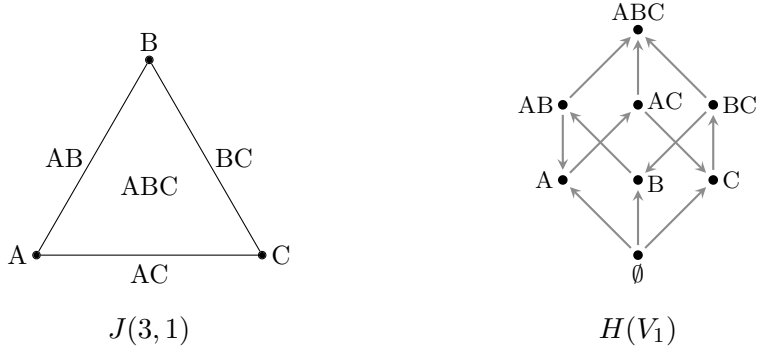

\begin{defn}
Consider the directed Hasse diagram $H(V)$ described in Remark \ref{rem:hasse}. If $H(V)$ has no directed cycles then we say that $V$ is an \textbf{acyclic matching} of the Hasse diagram of $K$. We call such a matching a \textbf{partial} matching if not every cell is paired with another and say it is a \textbf{complete} matching otherwise.
\end{defn}

\begin{thm}[\textbf{Forman}]
\label{thm:Forman}
Let $V$ be a discrete vector field on a regular CW complex $K$.
\begin{enumerate}
\item [\rm(i)] There are no nontrivial closed $V$-paths if and only if $V$ is an acyclic matching of the Hasse diagram of $K$.
\item [\rm(ii)] Suppose that $V$ is an acyclic partial matching of the Hasse diagram of $K$ in which the empty set is unpaired. Let $u_p$ denote the number of unpaired $p$-cells. Then $K$ is homotopic to a CW complex with exactly $u_p$ cells of dimension $p$ for each $p\geq 0$.
\end{enumerate}
\end{thm}

\begin{proof}
Part (i) is \cite[Theorem 6.2]{Forman02} and part (ii) is \cite[Theorem 6.3]{Forman02}.
\end{proof}

Finding a complete acyclic matching for the CW complex associated with the hypersimplex will be the focus of Section \ref{s:Construction}. Theorem \ref{thm:Forman} (i) makes it much easier to show that this matching will be acyclic and this will be the focus of Section \ref{s:acyclicproof}.

\begin{eg}
Consider the hypersimplex $J(3,1)$ and let $V_1$ be the discrete vector field as in Example \ref{eg:Hasse0}. It can be seen by using brute force that this is not an acyclic partial matching because there exist several directed cycles, for example, there is one that follows the path $A,AC,C,BC,B,AB,A$. We could also prove this by using Theorem \ref{thm:Forman} (i) since $A$, $AB$, $B$, $BC$, $C$, $AC$, $A$ is a nontrivial closed $V_1$-path. Notice that the second path is the same as the first except in reverse order. This is because a directed cycle follows the arrows, while a $V$-path goes against the arrows by definition.

On the other hand let $V_2$ be the discrete vector field made up of the following pairs of faces:
\begin{center}
$(\emptyset, A), (B, AB), (C, AC),$ and $(BC, ABC).$
\end{center}
This does turn out to be an acyclic matching and in this case it is even a complete matching. We can see that this is acyclic by once again checking for any cycles in the diagram in Figure \ref{fig:Hasse} through brute force. Proving this otherwise takes a little bit more work, as will be seen in Section \ref{s:acyclicproof}.

\end{eg}

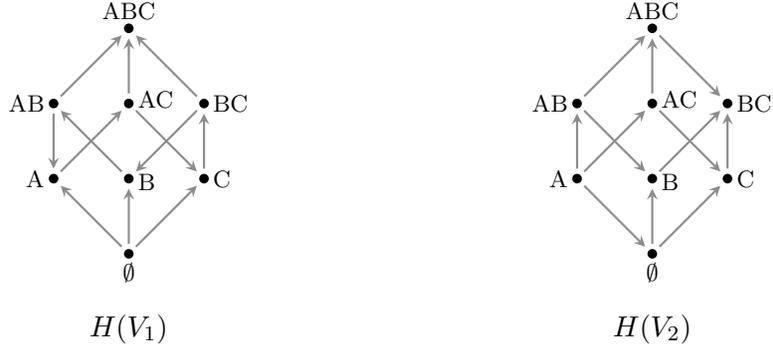
\begin{figure}[htbp]
\begin{center}
\begin{tikzpicture}[scale=1, >=stealth, thick, shorten >=2pt, shorten <=2pt]
\tikzset{shade t/.style ={gray!90}}
\tikzset{nodecirc/.style ={circle, draw, inner sep=1pt, fill=black}}
\node (e) at (0,0) [nodecirc] {};
\node (A) at (-1,1) [nodecirc] {};
\node (B) at (0,1) [nodecirc] {};
\node (C) at (1,1) [nodecirc] {};
\node (AB) at (-1,2) [nodecirc] {};
\node (AC) at (0,2) [nodecirc] {};
\node (BC) at (1,2) [nodecirc] {};
\node (ABC) at (0,3) [nodecirc] {};

\footnotesize
\node at (e) [below] {$\emptyset$};
\node at (A) [left] {A};
\node at (0,.95) [right] {B};
\node at (C) [right] {C};
\node at (AB) [left] {AB};
\node at (0,2.05) [right] {AC};
\node at (BC) [right] {BC};
\node at(ABC) [above] {ABC};
\normalsize

\draw [->] [shade t] (e) -- (A);
\draw [->] [shade t] (e) -- (B);
\draw [->] [shade t] (e) -- (C);
\draw [->] [shade t] (A) -- (AC);
\draw [<-] [shade t] (A) -- (AB);
\draw [->] [shade t] (B) -- (AB);
\draw [<-] [shade t] (B) -- (BC);
\draw [<-] [shade t] (C) -- (AC);
\draw [->] [shade t] (C) -- (BC);
\draw [->] [shade t] (AC) -- (ABC);
\draw [->] [shade t] (AB) -- (ABC);
\draw [->] [shade t] (BC) -- (ABC);

\node at (0,-1) {$H(V_1)$};

\end{tikzpicture}
\hspace{1.3 in}
\begin{tikzpicture}[scale=1, >=stealth, thick, shorten >=2pt, shorten <=2pt]
\tikzset{shade t/.style ={gray!90}}
\tikzset{nodecirc/.style ={circle, draw, inner sep=1pt, fill=black}}
\node (e) at (0,0) [nodecirc] {};
\node (A) at (-1,1) [nodecirc] {};
\node (B) at (0,1) [nodecirc] {};
\node (C) at (1,1) [nodecirc] {};
\node (AB) at (-1,2) [nodecirc] {};
\node (AC) at (0,2) [nodecirc] {};
\node (BC) at (1,2) [nodecirc] {};
\node (ABC) at (0,3) [nodecirc] {};

\footnotesize
\node at (e) [below] {$\emptyset$};
\node at (A) [left] {A};
\node at (0,.95) [right] {B};
\node at (C) [right] {C};
\node at (AB) [left] {AB};
\node at (0,2.05) [right] {AC};
\node at (BC) [right] {BC};
\node at(ABC) [above] {ABC};
\normalsize

\draw [<-] [shade t] (e) -- (A);
\draw [->] [shade t] (e) -- (B);
\draw [->] [shade t] (e) -- (C);
\draw [->] [shade t] (A) -- (AC);
\draw [->] [shade t] (A) -- (AB);
\draw [<-] [shade t] (B) -- (AB);
\draw [->] [shade t] (B) -- (BC);
\draw [<-] [shade t] (C) -- (AC);
\draw [->] [shade t] (C) -- (BC);
\draw [->] [shade t] (AC) -- (ABC);
\draw [->] [shade t] (AB) -- (ABC);
\draw [<-] [shade t] (BC) -- (ABC);

\node at (0,-1) {$H(V_2)$};

\end{tikzpicture}
\vspace{.3 in}
\caption{Example of a cyclic and an acyclic matching}
\label{fig:Hasse}
\end{center}
\end{figure}


\section{Construction of Matchings}
\label{s:Construction}

In this section we will define a family of matchings dependent on two integers $m_0$ and $m_1$, where $0\leq m_0\leq n-k-1$ and $1\leq m_1\leq k-1$. 
These two numbers will allow us to make adjustments to the matching relative to the faces $F(S)$ where $S(0)=m_0$ and $S(1)=m_1$. 

We start by matching the vertex $v_0 = (1,\ldots,1,0,\ldots,0)$ with $\emptyset$. 
Next we fix $m_0$ and $m_1$ and then match $F(S)$ with $F(S')$ where $S'$ is obtained from $S$ by doing one of the following replacements to $S$:

\begin{enumerate}
  \item If $S(1)$ $\leq k-1$ and there is a 1 to the right of the rightmost * and one of the following is true:
  \begin{enumerate}
    \item $S(1) \neq m_1$;
    \item $S(1) = m_1$ and $S(0)>m_0$; 
    \item $S(1) = m_1$, $S(0)=m_0$, and there is not a 0 to the left of the leftmost *;
  \end{enumerate}
then replace the rightmost 1 with a *.

  \item If $S(1)$ $\leq k-2$ and there is no 1 right of the rightmost *, and one of the following is true:
  \begin{enumerate}
    \item $S(1) \neq m_1 - 1$;
    \item $S(1) = m_1 - 1$ and $S(0)>m_0$;
    \item $S(1) = m_1 - 1$, $S(0)=m_0$, and there is not a 0 to the left of the leftmost *;
  \end{enumerate}
then replace the rightmost * with a 1. 
  \item If $S(1) = k-1$, $S(0)\leq n-k-1$, there is no 1 to the right of the rightmost *, and there is a 0 to the left of the leftmost *, then replace the leftmost 0 with a *.
  \item If $S(1) = k-1$, $S(0)\leq n-k-2$, there is no 1 to the right of the rightmost *, and there is no 0 to the left of the leftmost *, then replace the leftmost * with a 0. 
  \item If $S(1) = m_1$, $S(0)\leq m_0$, there is a 1 to the right of the rightmost *, and there is a 0 to the left of the leftmost *, then replace the leftmost 0 with a *.
  \item If $S(1) = m_1$, $S(0)<m_0$, there is a 1 to the right of the rightmost *, and there is no 0 to the left of the leftmost *, then replace the leftmost * with a 0. 
  \item If $S(1) = m_1 - 1$, $S(0)\leq m_0$, there is no 1 to the right of the rightmost *, and there is a 0 to the left of the leftmost *, then replace the leftmost 0 with a *.
  \item If $S(1) = m_1 - 1$, $S(0)<m_0$, there is no 1 to the right of the rightmost *, and there is no 0 to the left of the leftmost *, then replace the leftmost * with a 0. 
  \item If $S(1)=k$, $S(0)=n-k$, and $F(S)\neq v_0= (1,\ldots,1,0,\ldots,0)$, then replace the leftmost 0 and the rightmost 1 with a *.
  \item If $S(1)=k-1$, $S(0)=n-k-1$, there is no 1 to the right of rightmost *, and there is no 0 to the left of the leftmost *, then replace the leftmost * with a 0 and the remaining * with a 1.
 
\end{enumerate}

\begin{eg}
Let $n=8$, $k=3$, and consider the matching with $m_0=2$ and $m_1=1$, then:
\begin{center}
$F$(0100*0*1) and $F$(0100*0**) are matched by (1)(a) and (2)(a);
\\
$F$(0*00*0*1) and $F$(0*00*0**) are matched by (1)(b) and (2)(b);
\\
$F$(***0*0*1) and $F$(***0*0**) are matched by (1)(c) and (2)(c);
\\
$F$(0100*1*0) and $F$(*100*1*0) are matched by (3) and (4);
\\
$F$(0****0*1) and $F$(*****0*1) are matched by (5) and (6);
\\
$F$(0****0**) and $F$(*****0**) are matched by (7) and (8);
\\
$F$(10010100) and $F$(1*010*00) are matched by (9) and (10).
\end{center}
\end{eg}

We will later on refer back to these rules to describe faces. A face $F(S)$ of type 1 is one such that $S$ satisfies the conditions of (1)(a), (b), or (c), a face of type 3 is one such that $S$ satisfies the conditions of (3), and so on. 

\begin{rem}
\label{rem:matching}
There are a few things here worth pointing out. 
First of all, if $F(S)$ is a face of type 1--8 or 10, then $S(0)<n-k$ and $S(1)<k$ and hence $S$ contains at least two *'s. Therefore $F(S)$ is the convex hull of more than one point and cannot be a vertex.
Also notice that if we are given a face $F(S)$ then (1), (3), (5), and (7) replace a 0 or 1 in $S$ with a * and hence match this face to one of higher dimension. On the other hand (2), (4), (6), and (8) replace a * in $S$ with a 0 or 1 which matches the given face to one of lower dimension. Similarly (9) matches a vertex with an edge, which is a face of higher dimension, and (10) matches an edge to a vertex.

\end{rem}

\begin{lem}
\label{l:wd}
For any $m_0$ and $m_1$, the ten rules above partition the set of faces of $J(n,k)$, other than $\{v_0\}$, into subsets of faces of type $i$ where $1\leq i\leq10$.
\end{lem}
\begin{proof}
Let $F(S)$ be a face of $J(n,k)$ other than $\{v_0\}$.
Notice that if $S$ contains no *'s then it only satisfies the conditions of (9). If we assume that $S$ has a 1 to the right of the rightmost * then it either satisfies (1), (5), or (6). 
If $S(1)\neq m_1$ then $S$ must satisfy (1)(a). If $S(1)=m_1$ and $S(0)>m_0$ then $S$ must satisfy (1)(b). If $S(1)= m_1$, $S(0)\leq m_0$, and there is not a 0 to the left of the leftmost * then $S$ satisfies (1)(c) when $S(0)=m_0$ and (6) otherwise. If $S(1)= m_1$, $S(0)\leq m_0$, and there is a 0 to the left of the leftmost * then $S$ satisfies (5).

Next we assume that $S$ does not have a 1 to the right of the rightmost * and that $S(1)=k-1$ and so it either satisfies (3), (4), or (10). If there
is no 0 to the left of the leftmost * in $S$ then $S$ satisfies (4) when $S(0)\leq n-k-2$ and satisfies (10) when $S(0)=n-k-1$. If there is a 0 to the left of the leftmost * in $S$ then $S$ only satisfies (3).

Finally, we assume that $S$ does not have a 1 to the right of the rightmost * and that $S(1)<k-1$ and so it either satisfies (2), (7), or (8) since $m_1-1 < k-1$. If $S(1)\neq m_1-1$ then $S$ must satisfy (2)(a). If $S(1)= m_1-1$ and $S(0)>m_0$ then $S$ must satisfy (2)(b). If $S(1)= m_1-1$, $S(0)\leq m_0$, and there is not a 0 to the left of the leftmost * then $S$ satisfies (2)(c) when $S(0)=m_0$ and (8) otherwise. If $S(1)= m_1-1$, $S(0)\leq m_0$, and there is a 0 to the left of the leftmost * then $S$ satisfies (7).
\end{proof}


Now that we have a description of the matchings that we will use, we need to show that they satisfy the conditions of being a complete acyclic matching. If we let $V$ be a collection of pairs of faces matched together by the rules above, then showing that $V$ is complete amounts to showing that every face appears in exactly one pair in $V$. Showing that $V$ is acyclic will be shown in Section \ref{s:acyclicproof}. For the remainder of this section and the next we will assume that we have a fixed matching, or in other words, that $m_0$ and $m_1$ are fixed.

\begin{defn}
Let $i,j\in\mathbb{Z}$ such that $1\leq i,j\leq 10$ and suppose that rule $(i)$ and rule $(j)$ are two of the ten rules from Section \ref{s:Construction}. 
We say that rule $(i)$ and rule $(j)$ are \textbf{inverses} of each other if both of the following are true:
\begin{enumerate}
\item [(i)] If $F(S)$ is matched with $F(S')$ by $(i)$, then $F(S')$ is matched with $F(S)$ by $(j)$.
\item [(ii)] If $F(S')$ is matched with $F(S)$ by $(j)$ , then $F(S)$ is matched with $F(S')$ by $(i)$.
\end{enumerate}
\end{defn}

\begin{lem} Consider the ten rules from Section \ref{s:Construction}.
\label{l:inv1}
\begin{enumerate}
\item [\rm{(i)}] The rules (1)(a) and (2)(a) are inverses of each other.
\item [\rm{(ii)}] The rules (1)(b) and (2)(b) are inverses of each other.
\item [\rm{(iii)}] The rules (1)(c) and (2)(c) are inverses of each other.
\end{enumerate}
\end{lem}

\begin{proof}
First suppose that $S$ satisfies the conditions of (1)(a) and so $S(1)\neq m_1$, $S(1)\leq k-1$, and there is a 1 to the right of the rightmost * in $S$. Then by (1), $S'$ is obtained from $S$ by replacing the rightmost 1 with a *, so $S'(1) = S(1)-1 \neq m_1-1$, $S'(1) = S(1)-1 \leq k-2$, and there is no 1 to the right of the rightmost * in $S'$. Therefore $S'$ satisfies the conditions of (2)(a) and hence will be matched to the face $F(S'')$ where $S''$ is obtained from $S'$ by replacing the rightmost * with a 1. This gives us $S'' = S$ as desired.

Next suppose that $S$ satisfies the conditions of (1)(b) and so $S(1) = m_1$, $S(1)\leq k-1$, $S(0)>m_0$, and there is a 1 to the right of the rightmost * in $S$. Then by (1), $S'$ is obtained from $S$ by replacing the rightmost 1 with a *, so $S'(1) = S(1)-1 = m_1-1$, $S'(1) = S(1)-1 \leq k-2$, $S'(0) = S(0)>m_0$, and there is no 1 to the right of the rightmost * in $S'$. Therefore $S'$ satisfies the conditions of (2)(b) and hence will be matched to the face $F(S'')$ where $S''$ is obtained from $S'$ by replacing the rightmost * with a 1. This gives us $S'' = S$ as desired.

Next suppose that $S$ satisfies the conditions of (1)(c) and so $S(1) = m_1$, $S(1)\leq k-1$, $S(0) = m_0$, there is a 1 to the right of the rightmost * in $S$, and there is not a 0 to the left of the leftmost *. Then by (1), $S'$ is obtained from $S$ by replacing the rightmost 1 with a *, so $S'(1) = S(1)-1 = m_1-1$, $S'(1) = S(1)-1 \leq k-2$, $S'(0) = S(0) = m_0$, there is no 1 to the right of the rightmost * in $S'$, and there is not a 0 to the left of the leftmost *. Therefore $S'$ satisfies the conditions of (2)(c) and hence will be matched to the face $F(S'')$ where $S''$ is obtained from $S'$ by replacing the rightmost * with a 1. This gives us $S'' = S$ as desired.

What remains to be shown in this case is that if we have a face $F(S')$ of type 2 then there is some face $F(S)$ of type 1 that matches with it. Notice that $S'$ has no 1 to the right of the rightmost * and hence when we consider the sequence $S$ obtained by replacing the rightmost * in $S'$ with a 1, we see that $S$ satisfies the conditions of (1) and that $F(S)$ will be matched with $F(S')$ as desired.
\end{proof}

\begin{lem}
\label{l:inv3}
The rules (3) and (4) are inverses of each other.
\end{lem}
\begin{proof}
Suppose first that $S$ satisfies the conditions of (3) and so $S(1)=k-1$, $S(0)\leq n-k-1$, there is no 1 to the right of the rightmost *, and there is a 0 to the left of the leftmost *. Then by (3), $S'$ is obtained from $S$ by replacing the leftmost 0 with a *, so $S'(1) = S(1)=k-1$, $S'(0)= S(0)-1\leq n-k-2$, there is no 1 to the right of the rightmost *, and there is no 0 to the left of the leftmost *. Therefore $S'$ satisfies the conditions of (4) and hence will be matched to the face $F(S'')$ where $S''$ is obtained from $S'$ by replacing the leftmost * with a 0. This gives us $S'' = S$ as desired.

Now let $S$ satisfy the conditions of (4) and so $S(1)=k-1$, $S(0)\leq n-k-2$, there is no 1 to the right of the rightmost *, and there is no 0 to the left of the leftmost *. Then by (4), $S'$ is obtained from $S$ by replacing the leftmost * with a 0, so $S'(1) = S(1)=k-1$, $S'(0)= S(0)+1\leq n-k-1$, there is no 1 to the right of the rightmost *, and there is a 0 to the left of the leftmost *. Therefore $S'$ satisfies the conditions of (3) and hence will be matched to the face $F(S'')$ where $S''$ is obtained from $S'$ by replacing the leftmost 0 with a *. This gives us $S'' = S$ as desired.
\end{proof}

\begin{lem}
\label{l:inv5}
The rules (5) and (6) are inverses of each other and the rules (7) and (8) are inverses of each other.
\end{lem}
\begin{proof}
This proof is similar to that of Lemma \ref{l:inv3}.
\end{proof}

\begin{lem}
\label{l:inv9}
The rules (9) and (10) are inverses of each other.
\end{lem}
\begin{proof}
Suppose first that $S$ satisfies the conditions of (9) and so $S(1)=k$, $S(0)=n-k$, and $S\neq 1\ldots 10\ldots 0$. Then by (9), $S'$ is obtained from $S$ by replacing the leftmost 0 and the rightmost 1 with a *, so $S'(1)=S(1)-1=k-1$, $S'(0)=S(0)-1=n-k-1$, there is no 1 to the right of rightmost *, there is no 0 to the left of the leftmost *. Therefore $S'$ satisfies the conditions of (10) and hence will be matched to the face $F(S'')$ where $S''$ is obtained from $S'$ by replacing the leftmost * with a 0 and the remaining * with a 1. This gives us $S'' = S$ as desired.

Next let $S$ satisfy the conditions of (10) and so $S(1) = k-1$, $S(0) = n-k-1$, there is no 1 to the right of rightmost *, and there is no 0 to the left of the leftmost *. Then by (10), $S'$ is obtained from $S$ by replacing the leftmost * with a 0 and the rightmost * with a 1. 
If $F(S')=F(1\ldots 10\ldots 0)=\{v_0\}$ then that would mean the leftmost * in $S$ was to the right of the rightmost * in $S$ which is a contradiction. So $S'(1)=S(1)+1=k$, $S'(0)=S(0)+1=n-k$, and $F(S')\neq \{v_0\}$. Therefore $S'$ satisfies the conditions of (9) and hence will be matched to the face $F(S'')$ where $S''$ is obtained from $S'$ by replacing the leftmost 0 and the rightmost 1 with a *. This gives us $S''=S$ as desired.
\end{proof}

\begin{prop}
Fix $m_0$ and $m_1$ and let $V$ be the corresponding collection of pairs of matched faces. Every face appears in exactly one pair in $V$ and hence $V$ is a discrete vector field and furthermore, $V$ is a complete matching.
\end{prop}
\begin{proof}
Lemma \ref{l:wd} shows that every face is matched with at least one other face and
Lemmas \ref{l:inv1}--\ref{l:inv9} show that every face is matched to at most one other face. Remark \ref{rem:matching} points out that for every pair of faces, one face is a codimension 1 face of the other and the assertions follow.
\end{proof}


\section{Proof that the Matchings are Acyclic}
\label{s:acyclicproof}

First recall that $S$ is a sequence of 0's, 1's, and *'s. For notational purposes it will be helpful to break up $S$ into subsequences of only 0's and 1's and subsequences of only *'s. From now let $f_i$ denote either a sequence of 0's and 1's or $\emptyset$ where $1\leq i\leq n$. Also let $f'_i$ be a subsequence of $f_i$, let $0 \cdots 0$ denote a sequence of 0's or possibly $\emptyset$, and let $1 \cdots 1$ denote a sequence of 1's or possibly~$\emptyset$. 

\begin{eg}
Suppose we have an edge given by $F(S)$. Since $F(S)$ is an edge, $S$ has exactly two *'s and therefore we can write $F(S)=F(f_1*f_2*f_3)$. Suppose also that $1\in f_3$, then in order to help show the location of the rightmost 1 in $f_3$ we could also write $f_3 = f'_310 \cdots 0$ since it is only possible for either a sequence of 0's or $\emptyset$ to be to the right of the rightmost 1 in $f_3$.
\end{eg}

Next, following the language of Forman \cite{Forman02} introduced in Section \ref{s:Morse}, if $K$ is the CW complex formed by the faces of $J(n,k)$ then let $V$ be the discrete vector field on $K$ defined by the collection of pairs of matched faces. Recall that a $V$-path is a sequence of cells
\begin{center}
$a_0,b_0,a_1,b_1,a_2,\ldots,b_r,a_{r+1}$
\end{center}
such that for each $i = 0,\ldots,r,$ each of $a_i$ and $a_{i+1}$ is a codimension 1 face of $b_i$, each $(a_i,b_i)$ belongs to $V$ (hence $a_i$ is matched with $b_i$), and $a_i\neq a_{i+1}$ for all $0\leq i\leq r$. If $r\geq 0$, we call the $V$-path nontrivial, and if $a_0 = a_{r+1}$, we call the $V$-path closed.

\begin{defn}
\label{def:concat}
Given two $V$-paths $a_0,b_0,a_1,\ldots,b_r,a_{r+1}$ and $a'_0,b'_0,a'_1,\ldots,b'_s,a'_{s+1}$ such that $a_{r+1} = a'_0$, define their \textbf{concatenation} to be the following $V$-path: 
\begin{center}
$a_0,b_0,a_1,\ldots,b_r,a'_0,b'_0,a'_1,\ldots,b'_s,a'_{s+1}$.
\end{center}
\end{defn}

We will now show in Lemmas \ref{lem:acyc1}--\ref{lem:acyc7} that there are no nontrivial, closed $V$-paths when $a_0$ is not a vertex and deal with the case when it is a vertex in Lemma \ref{lem:acyc8}.

\begin{lem}
\label{lem:acyc1}
Given a $V$-path $a_0,b_0,a_1$ where $a_0 = F(S) = F(f_1*\cdots f_i*f_{i+1})$ is of type 1, either $a_1 = F(S')$ is of type 1, 2, 5, 6, or 10 and going from $S$ to $S'$ replaces the rightmost 1 with a 0 or $a_1 = F(S')$ is of type 2, 3, 4, 7, 8, or 10, going from $S$ to $S'$ replaces the rightmost 1 with a *, and this new * is the rightmost * in $S'$.
\end{lem}
\begin{proof}
In this case $f_{i+1} = f'_{i+1}10 \cdots 0$ and so 
\begin{center}
$a_0 = F(f_1*\cdots f_i*f'_{i+1}10 \cdots 0$) and $b_0 = F(f_1*\cdots f_i*f'_{i+1}*0 \cdots 0)$.
\end{center}
In order to choose $a_1$, we can replace any * except the rightmost with 0 or 1 or replace the rightmost * with 0 since $a_1\neq a_0$. 
Either way, $S'(1) \leq S(1) \leq k-1$ and so $a_1$ cannot be of type 9.
First consider the case when we replace the rightmost * with 0 making 
\begin{center}
$a_1 = F(f_1*\cdots f_i*f'_{i+1}00 \cdots 0) = F(S')$.
\end{center}

If we assume that $a_1$ is of type 7 or 8 then $S(0)=S'(0)-1<m_0$ and $S(1)=S'(1)+1=m_1$ and hence $a_0$ would not be of type 1 which is a contradiction. Also, $a_1$ cannot be of type 3 or 4 since $S'(1) < S(1) \leq k-1$ and so $a_1$ must be of type 1, 2, 5, 6, or 10. Either way the result from $S$ to $S'$ was that we replaced the rightmost 1 with a 0.

Next consider the case when we replace any * except the rightmost with a 0 or 1, then we will write $f_1*\cdots f_i*f'_{i+1}$ as $f_1\cdots f'_{i+1}$ (since we do not know and it will not matter which * was replaced) and hence
\begin{center}
$a_1 = F(f_1\cdots f'_{i+1}*0 \cdots 0)$.
\end{center}

Notice that there is no 1 to the right of the rightmost * and so $a_1$ is of type 2, 3, 4, 7, 8, or 10. Either way the result from $S$ to $S'$ was that we replaced the rightmost 1 with a *.
\end{proof}

\begin{lem}
\label{lem:acyc2}
Given a $V$-path $a_0,b_0,a_1$ where $a_0 = F(S) = F(f_1*\cdots f_i*f_{i+1})$ is of type 3, either $a_1 = F(S')$ is of type 4 or 10 and going from $S$ to $S'$ does not change the rightmost * or $a_1 = F(S')$ is of type 1, 4, 6, or 10 and going from $S$ to $S'$ replaces the rightmost * with a 0.
\end{lem}
\begin{proof}
In this case $f_1 = 1 \cdots 10f'_1$, $f_{i+1} = 0\cdots0$, and $S(1) = k-1$ and so 
\begin{center}
$a_0 = F(1 \cdots 10f'_1*\cdots f_i*0\cdots0)$ and $b_0 = F(1 \cdots 1*f'_1*\cdots f_i*0\cdots0)$.
\end{center}
In order to choose $a_1 = F(S')$ we can replace any * except the leftmost only with a 0 since $S(1) = k-1$. 
No matter what is replaced, $S(1) = S'(1) = k-1$ and so $a_1$ cannot be of type 2, 7, 8, or 9.

First consider the case when any * but the leftmost or rightmost is replaced by a 0, then 
\begin{center}
$a_1 = F(1 \cdots 1*f'_1\cdots f_i*0\cdots0)$.
\end{center}
In this case there is no 1 to the right of the rightmost~* in $S'$ and so $a_1$ cannot be of type 1, 5, or 6. Similarly there is no 0 to the left of the leftmost~* in $S'$ and so $a_1$ cannot be of type 3. Therefore $a_1$ is of type 4 or 10.
Also, going from $S$ to $S'$ does not change the rightmost~*.

Next consider the case when the rightmost * is replaced by a 0 and so
\begin{center}
$a_1 = F(1 \cdots 1*f'_1*\cdots *f_i00\cdots0)$.
\end{center}
There is no 0 to the left of the leftmost * in $S'$ and so $a_1$ is not of type 3 or 5. Therefore $a_1$ can only be of type 1, 4, 6, or 10. Either way going from $S$ to $S'$ replaces the rightmost * with a 0.
\end{proof}

\begin{lem}
\label{lem:acyc3}
Given a $V$-path $a_0,b_0,a_1$ where $a_0 = F(S) = F(f_1*\cdots f_i*f_{i+1})$ is of type 5, $a_1 = F(S')$ is of type 1 or 6.
\end{lem}
\begin{proof} 
In this case $f_1 = 1\cdots 10f'_1$ and $f_{i+1} = f'_{i+1}10\cdots 0$ and so we have
\begin{center}
$a_0 = F(1\cdots 10f'_1*\cdots*f'_{i+1}10\cdots 0)$ and $b_0 = F(1\cdots 1*f'_1*\cdots*f'_{i+1}10\cdots 0)$.
\end{center}
There are again two cases for choosing $a_1$, but either way $S'(0) \leq S(0) \leq m_0 \leq n-k-1$ and so $a_1$ cannot be of type 9. 
For the first case, replace any * except the leftmost by a 0. Since there is a 1 to the right of the rightmost *, $a_1$ cannot be of type 2, 3, 4, 7, 8, or 10. Also since there is not a 0 to the left of the leftmost *, $a_1$ cannot be of type 5 and so $a_1$ can be of type 1 or 6.

On the other hand if we replace any * by a 1 we again have that there is a 1 to the right of the rightmost * and so $a_1$ cannot be of type 2, 3, 4, 7, 8, or 10. However this time $S'(1) = S(1)+1 = m_1+1$ and so $a_1$ cannot be of type 5 or 6 either. Therefore $a_1$ must be of type 1.
\end{proof}

\begin{lem}
\label{lem:acyc4}
Given a $V$-path $a_0,b_0,a_1$ where $a_0 = F(S) = F(f_1*\cdots f_i*f_{i+1})$ is of type 7, one of three things can happen:
\begin{enumerate}
\item [\rm{(i)}] $a_1 = F(S')$ is of type 1, 2, or 8 and going from $S$ to $S'$ replaces the rightmost * with a 0; 
\item [\rm{(ii)}] $a_1 = F(S')$ is of type 2, 3, 4, 8, or 10 and going from $S$ to $S$ leaves the rightmost * unchanged; or
\item [\rm{(iii)}] $a_1 = F(S')$ is of type 6 and going from $S$ to $S'$ replaces the rightmost * with a 1.
\end{enumerate}

\end{lem}
\begin{proof} 
In this case $f_1 = 1\cdots 10f'_1$ and $f_{i+1} = 0\cdots 0$ and so we have
\begin{center}
$a_0 = F(1\cdots 10f'_1*\cdots f_i*0\cdots 0)$ and $b_0 = F(1\cdots 1*f'_1*\cdots f_i*0\cdots 0)$.
\end{center}
This time choosing $a_1$ will be broken up into three cases, but in all of them $S'(0) \leq S(0) \leq m_0 \leq n-k-1$ and so $a_1$ cannot be of type 9. In the first case replace the rightmost * by 0 and so $S'(1) = S(1) = m_1-1 \leq k-2$. This means $a_1$ cannot be of type 3, 4, 5, 6, or 10. Also there is no 0 to the left of the leftmost * and so $a_1$ cannot be of type 7 either. Therefore $a_1$ can be of type 1, 2, or 8.

For the second case either replace any * but the leftmost or rightmost by a 0 or replace any * but the rightmost by a 1. Either way going from $S$ to $S'$ leaves the rightmost * unchanged. Next note that there is no 1 to the right of the rightmost * and hence $a_1$ cannot be of type 1, 5, or 6. Also if any * but the leftmost or rightmost is replaced by a 0 then there is no 0 to the left of the leftmost *, whereas if any * but the rightmost is replaced by a 1 then $S'(1) = m_1$, so in either case $a_1$ cannot be of type 7. Therefore $a_1$ can be of type 2, 3, 4, 8, or 10.

Last of all, we can choose $a_1$ by replacing the rightmost * by a 1.
In this case $S'(1) = S(1)+1 = m_1$ and $S'(0)=S(1)-1<m_0$ so $a_1$ cannot be of type 1, 7, or 8. 
Since there is a 1 to the right of the rightmost * in $S'$, $a_1$ cannot be of type 2, 3, 4, or 10.
Since there is not a 0 to the left of the leftmost * in $S'$, $a_1$ cannot be of type 5 either. 
Therefore $a_1$ can only be of type 6.
\end{proof}

\begin{lem}
\label{lem:acyc5}
Let $a_0,b_0,a_1,\ldots,b_r,a_{r+1}$ be a $V$-path such that $a_0 = F(S)$ is of type 1 and $a_j = F(S^{(j)})$ for each $j\geq0$. If, while going from $S$ to $S^{(j)}$ for any $j\geq1$, the rightmost 1 of $S$ is replaced with a 0, then $a_0\neq a_{r+1}$.
\end{lem}
\begin{proof}
First notice that in order to get from $S^{(j)}$ to $S^{(j+1)}$ for any $j\geq0$ either a 0 or 1 in $S^{(j)}$ is changed to a * based on the ten rules for matching faces and then a * is replaced by a 0 or 1. 
Furthermore, every $a_i$ has the same dimension as $a_0$ by definition of a $V$-path and since $a_0$ is a face of type 1 and hence not a vertex, none of the $a_i$ are vertices. Therefore we will never utilize rule (9).

Recall that a face of type 1 looks like
\begin{center}
$a_0 = F(f_1*\cdots f_i*f'_{i+1}10\cdots 0)$.
\end{center}
If there are any 0's to the right of the rightmost 1 in $S$ then they remain unchanged while going from $S^{(j)}$ to $S^{(j+1)}$ for any $j\geq0$, since the only rules other than (9) that involve replacing a 0 with a * require that the 0 be to the left of the leftmost *. 

If the rightmost 1 of $S$ is replaced with a 0 while going from $S$ to $S^{(j)}$ for any $j\geq1$, then there are no *'s to the right of that 0 in $S^{(j)}$. Therefore that 0 will remain fixed while going from $S^{(j)}$ to $S^{(r+1)}$ because again, the only rules other than (9) that involve replacing a 0 with a * require that the 0 be to the left of the leftmost~* and so $a_0\neq a_{r+1}$.
\end{proof}

\begin{lem}
\label{lem:acyc6}
There are no nontrivial, closed $V$-paths $a_0,b_0,a_1,\ldots,b_r,a_{r+1}$ such that $a_0 = F(S)$ is of type 1.
\end{lem}
\begin{proof}
Suppose there is a nontrivial, closed $V$-path $a_0,b_0,a_1,\ldots,b_r,a_{r+1}$ with $a_0$ being of type 1. Then $a_{r+1} = a_0$ and hence is also of type 1. Note that if for any $j$, $a_j$ is of type 2, 4, 6, 8, or 10 then the $V$-path is not closed since faces of those types are paired with faces of lower dimension.

By Lemmas \ref{lem:acyc1}--\ref{lem:acyc4}, this $V$-path must look like one of the following: 
\begin{enumerate}
\item [(i)] $a_0,b_0,a_1$ where $a_1$ is of type 1,
\item [(ii)] $a_0,b_0,a_1,b_1,a_2$ where $a_1$ is of type 3 and $a_2$ is of type 1,
\item [(iii)] $a_0,b_0,a_1,b_1,a_2$ where $a_1$ is of type 5 and $a_2$ is of type 1,
\item [(iv)] $a_0,b_0,a_1,b_1,a_2$ where $a_1$ is of type 7 and $a_2$ is of type 1,
\item [(v)] $a_0,b_0,a_1,b_1,a_2,b_2,a_3$ where $a_1$ is of type 7, $a_2$ is of type 3, and $a_3$ is of type 1,
\end{enumerate}
or a concatenation of several of these in the sense of Definition \ref{def:concat}. 

If we let $F(S')$ be the last face in each $V$-path above, then 
we finish this proof by showing case by case that the rightmost 1 in $S$ is replaced by a 0 while going from $S$ to $S'$.
This implies $a_0\neq a_{r+1}$ by Lemma \ref{lem:acyc5}, contradicting the assumption that the $V$-path is closed.

If our $V$-path starts as in (i) or (iii), then Lemma \ref{lem:acyc1} shows that the rightmost 1 in $S$ is replaced by a 0 while going from $a_0$ to $a_1$. 

If our $V$-path starts as in (ii), then Lemma \ref{lem:acyc1} shows that the rightmost 1 in $S$ is replaced by a * while going from $a_0$ to $a_1$ and Lemma \ref{lem:acyc2} shows that this * is then replaced by a 0 while going from $a_1$ to $a_2$. 

If our $V$-path starts as in (iv), then Lemma \ref{lem:acyc1} shows that the rightmost 1 in $S$ is replaced by a * while going from $a_0$ to $a_1$ and Lemma \ref{lem:acyc4} shows that this * is then replaced by a 0 while going from $a_1$ to $a_2$. 

If our $V$-path starts as in (v), then Lemma \ref{lem:acyc1} shows that the rightmost 1 in $S$ is replaced by a * while going from $a_0$ to $a_1$; Lemma \ref{lem:acyc4} shows that this * is unchanged while going from $a_1$ to $a_2$; and finally Lemma \ref{lem:acyc2} shows that this * is then replaced by a 0 while going from $a_2$ to $a_3$.
\end{proof}

\begin{lem}
\label{lem:acyc7}
There are no nontrivial, closed $V$-paths $a_0,b_0,a_1,\ldots,b_r,a_{r+1}$ such that $a_0 = F(S)$ is of type 3, 5, or 7.
\end{lem}
\begin{proof}
Suppose there is a nontrivial, closed $V$-path $a_0,b_0,a_1,\ldots,b_r,a_{r+1}$ with $a_0$ being of type 3. Again note that if for any $j$, $a_j$ is of type 2, 4, 6, or 8 then the $V$-path is not closed since faces of those types are paired with faces of lower dimension. 
By Lemma \ref{lem:acyc2} we must have that $a_1$ is of type 1 and hence $a_1,\ldots,b_r,a_0,b_0,a_1$ is a nontrivial, closed $V$-path with $a_1$ being of type 1. 
This contradicts Lemma \ref{lem:acyc6}. 
A similar argument deals with the cases when $a_0$ is of type 5 (using Lemma \ref{lem:acyc3}) and when $a_0$ is of type 7 (using Lemma \ref{lem:acyc4}).
\end{proof}

\begin{lem}
\label{lem:acyc8}
There are no nontrivial, closed $V$-paths $a_0,b_0,a_1,\ldots,b_r,a_{r+1}$ such that $a_0 = F(S)$ is a vertex.
\end{lem}
\begin{proof}
If $a_0 = \{v_0\} = \{(1,\ldots,1,0,\ldots,0)\}$ then we are done since $\{v_0\}$ is not matched with an edge. If $a_0 = F(S)$ is any other vertex, then the rightmost 1 has at least one 0 to its left in $S$ and so
$$a_0 = F(1\cdots10f_110\cdots0), \ b_0 = F(1\cdots1*f_1*0\cdots0),$$
and
$$a_1 = F(1\cdots11f_100\cdots0) = F(S').$$
Notice the net result from $S$ to $S'$ was that the rightmost 1 was moved to its left.

Assume that $a_0,b_0,a_1,\ldots,b_r,a_{r+1}$ is a closed $V$-path such that $a_0$ is a vertex and that $a_i = F(S^{(i)})$ for $0\leq i\leq r+1$. For the same reasons as above none of the $a_i$ can be equal to $\{v_0\}$ and going from $S^{(i)}$ to $S^{(i+1)}$ always moves the rightmost 1 to its left and therefore $a_0 \neq a_{r+1}$, which is a contradiction.
\end{proof}

\begin{thm}
\label{thm:acyc}
The matchings described in Section \ref{s:Construction} are acyclic.
\end{thm}
\begin{proof}
This follows from Theorem \ref{thm:Forman} (i) which says that there are no nontrivial, closed $V$-paths if and only if $V$ is an acyclic matching of the Hasse diagram of $K$. We have already shown there are no nontrivial closed $V$-paths in Lemmas \ref{lem:acyc6}--\ref{lem:acyc8} and so we are done.
\end{proof}

\section*{Acknowledgements}
This material is based upon work supported by the National Science Foundation
under Grant Number 0905768.
Any opinions, findings, and conclusions or recommendations expressed in this
material are those of the author and do not necessarily reflect the views of the
National Science Foundation.

\bibliographystyle{amsplain}
\bibliography{refs}		
\end{document}